\documentclass[a4,11pt]{amsart} 

\usepackage{amsfonts, amsmath, lscape}
\usepackage[margin=3.5cm]{geometry}

\usepackage[utf8]{inputenc}
\usepackage{comment}
\usepackage[T1]{fontenc}
\usepackage[english]{babel}
\usepackage{nicematrix}
\usepackage{pifont}
\usepackage{enumitem}
\usepackage{amsthm}
\usepackage{amssymb}
\usepackage{hyperref}
\usepackage{tikz}
\usepackage{tikz-network}
\usetikzlibrary{shapes}
\usetikzlibrary{positioning}
\tikzstyle{vertex}=[circle,draw,inner sep=0pt,minimum size=6pt]
\newcommand{\vertex}{\node[vertex]}
\usepackage{csquotes}
\usepackage{listings}
\usepackage{float, lscape}

\newtheorem{theorem}{Theorem}[section]
\newtheorem{lemma}[theorem]{Lemma}
\newtheorem{proposition}[theorem]{Proposition}

\newtheorem{conjecture}[theorem]{Conjecture}

\usepackage{caption} 

\begin{document}

\title[Gluing methods for string C-groups]{Two gluing methods for string C-group representations of the symmetric groups}

\author{Dimitri Leemans}
\address{Dimitri Leemans, Département de Mathématique, Université libre de Bruxelles, C.P.216, Boulevard du Triomphe, 1050 Brussels, Belgiumi. Orcid number 0000-0002-4439-502X.}
\email{leemans.dimitri@ulb.be}

\author{Jessica Mulpas}
\address{Jessica Mulpas, Département de Mathématique, Université libre de Bruxelles, C.P.216, Boulevard du Triomphe, 1050 Brussels, Belgiumi. Orcid number 0000-0002-6128-2561.}
\email{jessica.mulpas@outlook.com}

\keywords{String C-group representations, symmetric groups, permutation representation graphs, CPR graphs}
\subjclass[2000]{20B30, 05C25, 52B15}

\begin{abstract}
The study of string C-group representations of rank at least $n/2$ for the symmetric group $S_n$ has gained a lot of attention in the last fifteen years.
In a recent paper, Cameron et al. gave a list of permutation representation graphs of rank $r\geq n/2$ for $S_n$, having a fracture graph and a non-perfect split. They conjecture that these graphs are permutation representation graphs of string C-groups. In trying to prove this conjecture, we discovered two new techniques to glue two CPR graphs for symmetric groups together.  We discuss the cases in which they yield new CPR graphs. By doing so, we invalidate the conjecture of Cameron et al.
We believe our gluing techniques will be useful in the study of string C-group representations of high ranks for the symmetric groups.
\end{abstract}

\maketitle

\section{Introduction}
In the last fifteen years, a lot of attention has been dedicated to string C-group representations of the symmetric and alternating groups. We refer to~\cite[Section 4]{PolBible} for a survey of these results until 2019.
Results obtained in several papers (see~\cite{CPRSn, CPRSnCorr, extSn, RankofPolAltGroups, highestrkan}) led Cameron, Fernandes and Leemans to prove that if $n$ is large enough, up to isomorphism and duality, the number of string C-groups of rank $r$ for $S_n$, with $r\geq (n+3)/2$, is the same as the number of string C-group representations of rank $r+1$ for $S_{n+1}$ (see~\cite[Theorem 1.1]{woof}).
In the process of proving that theorem, they determined all permutation representation graphs for $S_n$ with rank $r\geq n/2$ having a fracture graph and a split that is not perfect (see~\cite[Table 6]{woof}). They also conjectured that these graphs are permutation representation graphs of string C-groups.
In the process of trying to prove this conjecture, we discovered a gluing method which works as follows.

\begin{theorem}\label{taping}
    Suppose that the following two permutation representation graphs $\mathcal{G}$ and $\tilde{\mathcal{G}}$ (where $\mathcal{G}'$ and $\mathcal{G}''$ are subgraphs containing either no edge or only edges of labels at least $1$) are CPR graphs.

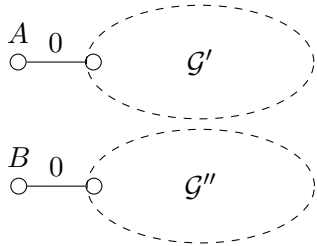
\begin{center}
\begin{tikzpicture}
\vertex[label=above:$A$] (1) at (0,0) {};
\vertex[] (2) at (1,0) {};
\node[draw, right=-6pt of 2, shape=ellipse, style={dashed, minimum width=86pt, minimum height=43pt}] (4) {$\mathcal{G}'$};
\draw[-] (1) -- (2)
 node[pos=0.5,above] {$0$};
\end{tikzpicture}
\end{center}

\begin{center}
\begin{tikzpicture}
\vertex[label=above:$B$] (1) at (0,0) {};
\vertex[] (2) at (1,0) {};

\node[draw, right=-6pt of 2, shape=ellipse, style={dashed, minimum width=86pt, minimum height=43pt}] (4) {$\mathcal{G}''$};
\draw[-] (1) -- (2)
 node[pos=0.5,above] {$0$};

\end{tikzpicture}
\end{center}

\noindent Then gluing vertex $A$ to vertex $B$ and relabeling the edges of the first graph by transforming every label $l$ into a label $-(l+1)$ gives a new CPR graph.
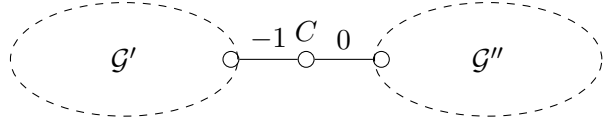
\begin{center}
\begin{tikzpicture}

\vertex[] (2) at (1,0) {};
\vertex[label=above:$C$] (3) at (2,0) {};
\vertex[] (4) at (3,0) {};

\node[draw, left=-6pt of 2, shape=ellipse, style={dashed, minimum width=86pt, minimum height=43pt}] (6) {$\mathcal{G}'$};
\node[draw, right=-6pt of 4, shape=ellipse, style={dashed, minimum width=86pt, minimum height=43pt}] (7) {$\mathcal{G}''$};
\draw[-] (3) -- (2)
 node[pos=0.5,above] {$-1$};
\draw[-] (3) -- (4)
 node[pos=0.5,above] {$0$};

\end{tikzpicture}
\end{center}

\end{theorem}

We give the proof of this theorem in Section~\ref{works}.
We also discovered what we thought would be another successful gluing method (see Conjecture~\ref{false}) which we found out does not work in all cases. This permits us to show that at least one graph of~\cite[Table 6]{woof}, namely graph (X), is not a permutation representation graph of a string C-group, and therefore that the conjecture of Cameron, Fernandes and Leemans in~\cite[Section 5.2]{woof} is false.


\section{Preliminaries}\label{prelim}
\subsection{String C-groups}

A {\em string C-group representation} (or string C-group for short) is a pair $(G,S)$ with $S:=\{\rho_0,\rho_1,...,\rho_{r-1}\}$ an ordered set of involutions that generates the group $G$, satisfying the following two properties.
\begin{description}[style=unboxed,leftmargin=0cm]
\item[(SP)] 
the {\em string property}, that is $(\rho_i\rho_j)^2=1_G$ for all $i,j \in \{ 0,1,...,r-1 \}$ with $\mid i-j \mid \geq 2$;

\item[(IP)] the {\em intersection property}, that is $\langle \rho_i \mid i \in I \rangle \cap \langle \rho_j \mid j \in J \rangle = \langle \rho_k \mid k \in I \cap J \rangle$ for any $I, J \subseteq \{ 0, 1,...,r-1 \}$.
\end{description}
When $(G,S)$ only satisfies the string property it is called a {\em string group generated by involutions} (or {\em sggi} for short). The {\em rank} of $(G,S)$ is the size of $S$.

Note that, from the definition above, one can  observe that string C-groups are smooth quotients of Coxeter groups with string diagrams. It is also a well-known fact that string C-groups are in one-to-one correspondence with  abstract regular polytopes, the latter being equivalent geometric formulations of the former (see~\cite[Section 2E]{ARP}).

For any subset $I\subseteq \{0, \ldots, r-1\}$, we denote $\langle \rho_j : j \in I\rangle$ by $G_I$. When $G$ is a string C-group, so is $G_I$.
If $\vert I\vert = r-1$ or $r-2$ then $\{0, \ldots, r-1\}\setminus I = \{i\}$ or $\{i,j\}$ (for some $i,j\in\{0,...,r-1\}$) and we denote $G_I$ by $G_i$ or $G_{i,j}$, respectively.

The {\em Schl\"afli type} of a string C-group representation $(G,\{\rho_0, \ldots, \rho_{r-1}\})$ is the ordered set $\{p_1,p_2,...,p_{r-1} \}$ where $p_i$ is the order of the element $\rho_{i-1}\rho_{i}$ for $i=1, \ldots, r-1$.

The {\em dual} of a string C-group representation $(G,\{\rho_0,\rho_1,...,\rho_{r-1}\})$ is the string C-group representation $(G,\{\rho_{r-1},\rho_{r-2},...,\rho_{0}\})$.

As in~\cite{RankofPolAltGroups}, for a sggi $(G, \{\rho_0,\rho_1,...,\rho_{r-1}\})$, an involution $\tau$ in a supergroup of $G$ such that $\tau\notin G$ and a fixed $k\in\{0,1,...,r-1\}$, one can define a new sggi $(G^*,S^*)$ where $S^* := \{\rho_i \tau^{\delta_{i,k}}\vert i\in\{0,1,...,r-1\}\}$ and $G^* := \langle S^*\rangle$
that we call the {\em sesqui-extension} of $G$ with respect to $\rho_k$ and $\tau$ (or {\em $k$-sesqui-extension} of $G$ with respect to $\tau$).
We have the following result.

\begin{proposition} \cite[Lemma 5.4]{RankofPolAltGroups}\label{sesqui}
If $G=\langle \rho_0,\rho_1,...,\rho_{r-1}\rangle$ and $G^*=\langle \rho_i \tau^{\delta_{i,k}}\vert i\in\{0,1,...,r-1\}\rangle$ is a $k$-sesqui extension of $G$ then $G^*\cong G$ or $G^*\cong G\times C_2$ (when $\tau\in G^*$).
\end{proposition}

It is not true to say that all sesqui-extensions of string C-groups are string C-groups but sesqui-extensions of string C-groups with respect to their first (or last) involutory generator do satisfy the intersection property as stated in the following proposition.

\begin{proposition}\cite[Proposition 5.3]{RankofPolAltGroups}\label{extremesesqui}
Any sesqui-extension $G^*$ of a string C-group $G$ with respect to $\rho_0$ is a string C-group. 
\end{proposition}

The following proposition is very useful to check if an sggi is a string C-group.
\begin{proposition}\cite[Proposition 2E16(a)]{ARP}\label{breakinparts}
An sggi $(G,\{\rho_0,...,\rho_{r-1}\})$ is a string C-group (that is if it satisfies the intersection property) if and only if 
\noindent  \begin{itemize}
 \item[$\bullet$] $(G_{0},\{\rho_1,...,\rho_{r-1}\})$ is a string C-group,
 \item[$\bullet$] $(G_{r-1},\{\rho_0,...,\rho_{r-2}\})$ is a string C-group, and
 \item[$\bullet$] $G_{0}\cap G_{r-1} = G_{0,r-1}:=\langle\rho_1,...,\rho_{r-2}\rangle$.
\end{itemize}

\end{proposition}

\subsection{Permutation representation graphs and CPR graphs}

Let $(G, \Omega)$ (with $\Omega := \{1,...,n\}$) be a permutation group generated by $r$ involutions $\rho_0,\rho_1,...,\rho_{r-1}$.
We define the {\em permutation representation graph} of $G$ to be the edge-labeled undirected multigraph $\mathcal{G}$ with $\Omega$ as vertex set and an edge with label $i$ between vertices $a$ and $b$ whenever $a\neq b$ and $\rho_i(a)=b$ (in which case, of course, $\rho_i(b)=a$).

When $(G,\{\rho_0,\rho_1,...,\rho_{r-1}\})$
is a string C-group, $\mathcal{G}$ is called a {\em CPR graph} as in~\cite{CPR}.

The following lemma gives the possible shapes of connected components of subgraphs of a CPR graph when looking at pairs of labels that are not consecutive.

\begin{lemma}\cite[Proposition 3.5]{CPR}\label{cprlem}
Each connected component of the subgraph of a CPR graph $\mathcal{G}$ induced by edges of labels $i$ and $j$ for $\vert i-j\vert \geq 2$ (i.e. by edges of labels $i$ and $j$ where $\rho_i$ and $\rho_j$ commute) is either a single vertex, a single edge, a double edge or an alternating square.
\end{lemma}

Observe that the proof of this lemma only requires $\mathcal G$ to be the permutation representation graph of an sggi.
Trivially, two generators commute unless their corresponding edges are adjacent.

Let us also notice that since $S_n$ is a transitive group in its natural action, any representation of $S_n \curvearrowright \{1,...,n\}$ by a permutation representation graph will be connected. Here we use $G\curvearrowright S$ to denote the action of a group $G$ on a set $S$.

Let $\mathcal{G}$ be the permutation representation of a string group generated by involutions $(G,\{\rho_0,\rho_1,...,\rho_{r-1}\})$ seen as a permutation group acting on a set $\Omega:=\{1, \ldots, n\}$. 
For $i,j\in\{0,1,...,r-1\}$ we define $\mathcal{G}_{i,j}$ to be the subgraph of $\mathcal{G}$ with vertex set $\Omega$ and all the edges in $\mathcal{G}$ except the ones of labels $i$ and $j$.

With the notations introduced in the previous section, we observe that $\mathcal{G}_{i,j}$ is the permutation representation graph for $G_{i,j}$ and that its connected components are the orbits of $G_{i,j}\curvearrowright \Omega$.

\subsection{Fracture graphs and perfect splits}

The notion of fracture graphs first appears in \cite{extSn}. For permutation groups (or permutation representation graphs) with a fracture graph, Cameron, Fernandes and Leemans then introduce the notion of splits and, in particular, perfect splits, in \cite{woof}.

Let $G=\langle \rho_0,\rho_1,...,\rho_{r-1}\rangle$ be an sggi that acts  faithfully on a set  $\{1,...,n\}$. We say that $G$ (or its permutation representation graph) \textit{has a fracture graph} if, for every $i\in\{0,1,...,r-1\}$, the subgroup $G_i$ has at least one more orbit than $G$ on $\{1,...,n\}$. Indeed, choosing, for each $i\in\{0,1,...,r-1\}$, two points $a_i$ and $b_i$ of $\{1,...,n\}$ permuted by $\rho_i$ and in different $G_i$-orbits, one obtains a transposition $(a_i,b_i)$ of $\rho_i$ and an edge $\{a_i,b_i\}$ of the permutation representation graph $\mathcal{G}$ of $G$. The graph with vertex set $\{1,...,n\}$ and edge set $\{\{a_i,b_i\}\vert i\in\{0,1,...,r-1\}\}$ is a spanning forest of $\mathcal{G}$ that we call \textit{fracture graph} of $G$ (or of $\mathcal{G}$).

Suppose that $G$ (or $\mathcal{G}$) has a fracture graph. Suppose also that, for some $i\in\{0,1,...,r-1\}$,  $G_i$ has exactly one more orbit than $G$ on $\{1,...,n\}$.  Let $\{\mathcal{O}_1,\mathcal{O}_2\}$ be a partition of $\{1,...,n\}$ such that $\rho_i$ is the unique generator of $G$ that swaps elements of $\mathcal{O}_1$ with elements of $\mathcal{O}_2$ and suppose that there is exactly one pair of points $(a,b)\in\mathcal{O}_1\times\mathcal{O}_2$ such that $\rho_i(a)=b$. Then we say that the edge $\{a,b\}$ of $\mathcal{G}$ is an \textit{$i$-split} of $G$ (or $\mathcal{G}$).

In this case, for all $j\neq i$, one can write $\rho_j$ as $\alpha_j\beta_j$ where $\alpha_j$ fixes $\mathcal{O}_2$ pointwise and $\beta_j$ fixes $\mathcal{O}_1$ pointwise. Similarly, $\rho_i=\alpha_i\beta_i(a,b)$ where $\alpha_i$ fixes $\mathcal{O}_2$ pointwise and $\beta_i$ fixes $\mathcal{O}_1$ pointwise. We can define $J_A:=\{j\in\{0,1,r-1\}\setminus\{i\}\vert \alpha_j\neq 1_{G}\}$ and similarly $J_B:=\{j\in\{0,1,r-1\}\setminus\{i\}\vert \beta_j\neq 1_{G}\}$ so that $A:=\langle\alpha_j\vert j\in J_A\rangle$ is the group induced by $G_i$ on $\mathcal{O}_1$ and $B:=\langle\beta_j\vert j\in J_B\rangle$ is the group induced by  $G_i$ on $\mathcal{O}_2$.

If no $j\in\{i+1,...,r-1\}$ is in $J_A$ and no $j\in\{0,1,...,i-1\}$ is in $J_B$ then we say that the $i$-split $\{a,b\}$ is \textit{perfect}.

\begin{proposition}\label{perfectsplit}\cite[Proposition 5.1]{woof}
If $G$ is transitive and has a perfect split then $G$ is primitive.
\end{proposition}
\subsection{Primitive groups}
The following lemma gives a criterion for a primitive permutation group to be isomorphic to a symmetric group.
\begin{lemma}\label{showsym}
    Let $G$ be a primitive permutation group of finite degree $n$ containing a transposition fixing at least three points. Then $G\cong S_n$.
\end{lemma}
\begin{proof}
\cite[Theorem 3.3E]{dixon1996permutation} and \cite[Theorem 1.1]{jones2014primitive} state that any primitive permutation group of finite degree $n$ containing a cycle of prime length fixing at least three points contains $A_n$ as a subgroup. Since a transposition is a $2$-cycle, $G\geq A_n$ and, in particular, since $G$ contains a transposition, $G\cong S_n$.
\end{proof}
\begin{theorem}\label{wreath}\cite[Theorem 1.8]{cameron1999permutation}
    Let $G$ be a primitive permutation group on $\Omega$. Let $\mathcal{B}=\{\mathcal{B}_1,...,\mathcal{B}_n\}$ be a complete system of imprimitivity for $G\curvearrowright\Omega$. Let $H$ be the group induced on $\mathcal{B}_1$ by its setwise stabilizer and let $K$ be the group induced on $\mathcal{B}$ by $G$. Then $G\leq H\wr K$.
\end{theorem}

\section{A successful gluing method}\label{works}

The main point of this section is to prove Theorem~\ref{taping}.
The permutation representation graphs that can be obtained by applying the gluing procedure described in Theorem \ref{taping} can only represent some groups closely related to symmetric groups. Let us first give a lemma to clarify the matter.

\begin{lemma}\label{howact1}
    Let $G=\langle \rho_0,\rho_1,...,\rho_{r-1}\rangle$ be a  string group generated by involutions represented by a permutation representation graph $\mathcal{G}$ of the following shape

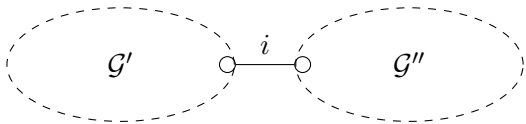
\begin{center}
\begin{tikzpicture}
\vertex[] (1) at (0,0) {};
\vertex[] (2) at (1,0) {};
\node[draw, left=-6pt of 1, shape=ellipse, style={dashed, minimum width=86pt, minimum height=43pt}] (6) {$\mathcal{G}'$};
\node[draw, right=-6pt of 2, shape=ellipse, style={dashed, minimum width=86pt, minimum height=43pt}] (7) {$\mathcal{G}''$};
\draw[-] (1) -- (2)
 node[pos=0.5,above] {$i$};
\end{tikzpicture}
\end{center}

\noindent where $\mathcal{G}'$ and $\mathcal{G}''$ are subgraphs of $\mathcal{G}$ whose edges (if there are any) are all of labels respectively at most $i-1$ and at least $i+1$.

\noindent If $\mathcal{G}$ is connected then $G$ is isomorphic to $S_n$ where $n$ is the size of the vertex set of $\mathcal{G}$.

\noindent If $\mathcal{G}$ is disconnected and $H$ is the group induced by the action of $G$ on the vertices of its connected components containing no $i$-edge then $G\cong S_t \times H$ where $t$ is the number of vertices in the connected component of $\mathcal{G}$ that contains the only $i$-edge.
\end{lemma}

\begin{proof}
First suppose that $\mathcal{G}$ is connected.

If $\mathcal{G}$ has at most four vertices then $i=0$ or $1$, $r=1$, $2$ or $3$ and we have one of the following three graphs.

\begin{center}
\begin{tikzpicture}
\vertex[] (8) at (-3,0) {};
\vertex[] (9) at (-2,0) {};
\draw[-] (8) -- (9)
 node[pos=0.5,above] {$0$};

\vertex[] (1) at (0,0) {};
\vertex[] (2) at (1,0) {};
\vertex[] (3) at (2,0) {};
\draw[-] (1) -- (2)
 node[pos=0.5,above] {$0$};
\draw[-] (3) -- (2)
 node[pos=0.5,above] {$1$};

\vertex[] (4) at (4,0) {};
\vertex[] (5) at (5,0) {};
\vertex[] (6) at (6,0) {};
\vertex[] (7) at (7,0) {};
\draw[-] (4) -- (5)
 node[pos=0.5,above] {$0$};
\draw[-] (5) -- (6)
 node[pos=0.5,above] {$1$};
\draw[-] (6) -- (7)
 node[pos=0.5,above] {$2$};
\end{tikzpicture}
\end{center}

These three graphs trivially represent generating sets for $S_2\cong C_2$, $S_3$ and $S_4$ respectively.

Now let us suppose that $\mathcal{G}$ has at least five vertices. Since $\rho_i$ fixes all the vertices that are not adjacent to the only $i$-edge of $\mathcal{G}$, it fixes, in this case, at least three vertices.
Note that this unique $i$-edge of $\mathcal{G}$ is a perfect split. Since $\mathcal{G}$ is connected, $G$ also acts transitively on its vertex set. Hence, by Proposition \ref{perfectsplit}, the action of $G$ is primitive. Therefore, since $G$ contains a cycle of prime length two fixing at least three vertices of $\mathcal{G}$, by Lemma \ref{showsym}, $G$ must be isomorphic to the symmetric group of degree $n:=\vert V(\mathcal G)\vert$.

Suppose that $\mathcal{G}$ is disconnected. Let $\mathcal{C}$ be its only connected component containing an $i$-edge and let $t$ be the number of vertices in $\mathcal{C}$.
By the above argument, $G$ acts as $S_t$ on the vertices of $\mathcal{C}$ and, moreover, since $\rho_i$ is a transposition, any transposition of $G$ that swaps two vertices of $\mathcal{C}$ can be written as ${\rho_i}^{\sigma}$ for some $\sigma \in G$. Since the other connected components of $\mathcal{G}$ have no $i$-edge, these elements generate the symmetric group on the vertices of $\mathcal{C}$ but fix any vertex of the other components of $\mathcal{G}$ whence $G\cong S_t \times H$ where $H$ is the group induced by the action of $G$ on the set of vertices that are not in $\mathcal C$.
\end{proof}

The proof of Theorem \ref{taping} needs a double induction, with the following lemma covering the base cases.

\begin{lemma}\label{basecase1}
For any CPR graph $\mathcal{G}$ of the following shape

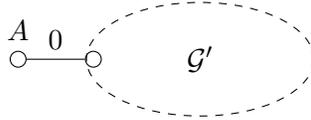
\begin{center}
\begin{tikzpicture}
\vertex[label=$A$] (1) at (0,0) {};
\vertex[] (2) at (1,0) {};
\node[draw, right=-6pt of 2, shape=ellipse, style={dashed, minimum width=86pt, minimum height=43pt}] (4) {$\mathcal{G}'$};
\draw[-] (1) -- (2)
 node[pos=0.5,above] {$0$};
\end{tikzpicture}
\end{center}

\noindent where $\mathcal{G}'$ is a subgraph of $\mathcal{G}$ whose edges (if there are any) all have labels at least $1$, the following permutation representation graph $\mathcal{H}$ obtained from connecting a $(-1)$-edge to vertex $A$

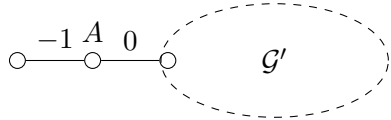
\begin{center}
\begin{tikzpicture}
\vertex[] (0) at (-1,0) {};
\vertex[label=$A$] (1) at (0,0) {};
\vertex[] (2) at (1,0) {};
\node[draw, right=-6pt of 2, shape=ellipse, style={dashed, minimum width=86pt, minimum height=43pt}] (4) {$\mathcal{G}'$};
\draw[-] (1) -- (2)
 node[pos=0.5,above] {$0$};
\draw[-] (1) -- (0)
 node[pos=0.5,above] {$-1$};
\end{tikzpicture}
\end{center}

\noindent is a CPR graph. Moreover, if they are connected, these graphs both represent symmetric groups.
\end{lemma}

\begin{proof}

It is clear from Lemma \ref{howact1} that both these graphs represent symmetric groups if they are connected.

Let $G:=\langle \rho_0,...,\rho_{r-1}\rangle$ be the string C-group represented by $\mathcal{G}$.
Let us show that $\mathcal{H}$ is a CPR graph by induction on the rank $r\geq 2$ of $G$.
To show that $\mathcal{H}$ is a CPR graph, we show that the group $\Gamma=\langle \rho_{-1},\rho_0,...,\rho_{r-1}\rangle$ it represents is a string C-group.
Note that, by Proposition \ref{cprlem} and because $G$ is an sggi, so is $\Gamma$.

If $r=1$ then $\Gamma$ is of rank two and it is clearly a string C-group.

Suppose $r>1$ and our construction yields a CPR graph starting from any CPR graph of rank $r-1$.
By Proposition \ref{breakinparts}, it is sufficient to show that $\Gamma_{-1}$ and $\Gamma_{r-1}$ are string C-groups and that $\Gamma_{-1}\cap\Gamma_{r-1}=\Gamma_{-1,r-1}$. 

Since $\Gamma_{-1}$ is isomorphic to $G$ (as sggi), it is clearly a string C-group.

Now $G_{r-1}$ is a string C-group of rank $r-1$ and thus $\Gamma_{r-1}$ is a string C-group by induction hypothesis (since its permutation representation graph is obtained from applying our construction to the one of $G_{r-1}$).

We are left with showing that $\Gamma_{-1}\cap\Gamma_{r-1}=\Gamma_{-1,r-1}$.
Obviously, $\Gamma_{-1,r-1}\leq \Gamma_{-1}\cap\Gamma_{r-1}$. We have equality if and only if the action of $\Gamma_{-1}\cap\Gamma_{r-1}$ on the orbits of $\Gamma_{-1,r-1}$ is the same as the one of $\Gamma_{-1,r-1}$. The orbits of $\Gamma_{-1,r-1}$ as acting on the vertex set of $\mathcal{H}$ are $\{B\}$, where $B=\rho_{-1}(A)$, $\mathcal{O}:=A^{\Gamma_{-1,r-1}}$ and sets of vertices $\mathcal{O}_1,...,\mathcal{O}_k$ ($k\in\mathbb{N}$) connected in $\mathcal{H}$ by edges of labels $1,...,r-2$.
Let $H$ be the group induced by the action of $\Gamma_{-1,r-1}$ on $\mathcal{O}_1\cup...\cup\mathcal{O}_k$.
By Lemma \ref{howact1}, the group induced by $\Gamma_{-1,r-1}$ on $\mathcal{O}$ is $Sym(\mathcal{O})$ and independent from $H$: we have $\Gamma_{-1,r-1}\cong S_t\times H$ where $t=\vert \mathcal{O}\vert$.
Clearly, since $\Gamma_{-1}$ fixes $B$, so does $\Gamma_{-1}\cap\Gamma_{r-1}$.
The vertices of $\mathcal{O}_1\cup...\cup\mathcal{O}_k$ are never connected by a $(-1)$-edge. Hence $\Gamma_{-1}\cap\Gamma_{r-1}$ cannot act as a proper overgroup of $H$ on $\mathcal{O}_1\cup...\cup\mathcal{O}_k$.
Moreover, by Lemma \ref{howact1}, both $\Gamma_{-1}$ and $\Gamma_{r-1}$ act as symmetric groups on their respective orbit containing $A$ (that contains the whole of $\mathcal{O}$) and independently on the others.
Therefore, $\Gamma_{-1}\cap\Gamma_{r-1}$ acts in the same way as $\Gamma_{-1,r-1}$ on the orbits of the latter and our claim holds.
\end{proof}

Note that this lemma does not always hold when $\mathcal{G}'$ has edges of label $0$. We give two basic counterexamples.

Consider the following permutation representation graph.

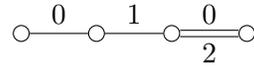
\begin{center}
\begin{tikzpicture}
\vertex[] (1) at (0,0) {};
\vertex[] (2) at (1,0) {};
\vertex[] (3) at (2,0) {};
\vertex[] (4) at (3,0) {};
\draw[-] (1) -- (2)
 node[pos=0.5,above] {$0$};
\draw[-] (2) -- (3)
 node[pos=0.5,above] {$1$};
\draw[thin,double distance=2pt] (4) -- (3)
 node[pos=0.5,above] {$0$} node[pos=0.5,below] {$2$};
\end{tikzpicture}
\end{center}

\noindent This is a CPR graph for the symmetric group $S_4$ as generated by $\{(1,2)(3,4), (2,3), (3,4)\}$ while the group $\Gamma=\langle \rho_{-1},\rho_0,\rho_1,\rho_2\rangle\cong S_5$ with permutation representation graph

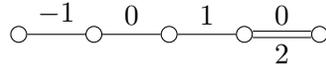
\begin{center}
\begin{tikzpicture}
\vertex[] (1) at (0,0) {};
\vertex[] (2) at (1,0) {};
\vertex[] (3) at (2,0) {};
\vertex[] (4) at (3,0) {};
\vertex[] (5) at (4,0) {};
\draw[-] (1) -- (2)
 node[pos=0.5,above] {$-1$};
\draw[-] (2) -- (3)
 node[pos=0.5,above] {$0$};
\draw[-] (4) -- (3)
 node[pos=0.5,above] {$1$};
\draw[thin,double distance=2pt] (4) -- (5)
 node[pos=0.5,above] {$0$} node[pos=0.5,below] {$2$};
\end{tikzpicture}
\end{center}

\noindent does not satisfy the intersection property. Indeed, with our usual notations, we have $\Gamma_{-1}\cap\Gamma_2>\Gamma_{-1,2}$ since both $\Gamma_{-1}$ and $\Gamma_2$ act as $S_4$ on the four rightmost vertices (in fact, $\Gamma_2\cong S_5$) and $\Gamma_{-1,2}$ clearly acts imprimitively on them.

Another valid counterexample would be the CPR graph

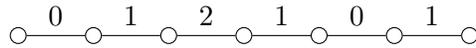
\begin{center}
\begin{tikzpicture}
\vertex[] (1) at (0,0) {};
\vertex[] (2) at (1,0) {};
\vertex[] (3) at (2,0) {};
\vertex[] (4) at (3,0) {};
\vertex[] (5) at (4,0) {};
\vertex[] (6) at (5,0) {};
\vertex[] (7) at (6,0) {};
\draw[-] (1) -- (2)
 node[pos=0.5,above] {$0$};
\draw[-] (2) -- (3)
 node[pos=0.5,above] {$1$};
\draw[-] (4) -- (3)
 node[pos=0.5,above] {$2$};
\draw[-] (4) -- (5)
 node[pos=0.5,above] {$1$};
\draw[-] (5) -- (6)
 node[pos=0.5,above] {$0$};
\draw[-] (7) -- (6)
 node[pos=0.5,above] {$1$};
\end{tikzpicture}
\end{center}

giving the permutation representation graph

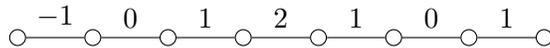
\begin{center}
\begin{tikzpicture}
\vertex[] (0) at (-1,0) {};
\vertex[] (1) at (0,0) {};
\vertex[] (2) at (1,0) {};
\vertex[] (3) at (2,0) {};
\vertex[] (4) at (3,0) {};
\vertex[] (5) at (4,0) {};
\vertex[] (6) at (5,0) {};
\vertex[] (7) at (6,0) {};
\draw[-] (1) -- (0)
 node[pos=0.5,above] {$-1$};
\draw[-] (1) -- (2)
 node[pos=0.5,above] {$0$};
\draw[-] (2) -- (3)
 node[pos=0.5,above] {$1$};
\draw[-] (4) -- (3)
 node[pos=0.5,above] {$2$};
\draw[-] (4) -- (5)
 node[pos=0.5,above] {$1$};
\draw[-] (5) -- (6)
 node[pos=0.5,above] {$0$};
\draw[-] (7) -- (6)
 node[pos=0.5,above] {$1$};
\end{tikzpicture}
\end{center}

that is not a CPR graph as $C_2\times C_2 \cong \Gamma_{-1} \cap \Gamma_{1,2} > \Gamma_{-1,1,2}\cong C_2$.

We now move onto the main proof of this section, the proof of Theorem \ref{taping}.

\begin{proof}[Proof of Theorem \ref{taping}]
Let $G=\langle \tilde{\rho}_0,...,\tilde{\rho}_{r-1}\rangle$ be the rank-$r$ string C-group represented by $\mathcal{G}$ and let $\tilde{G}=\langle \rho_0,...,\rho_{\tilde{r}-1}\rangle$ be the rank-$\tilde{r}$ string C-group represented by $\tilde{\mathcal{G}}$. Let $\mathcal{H}$ be the permutation representation graph obtained by our construction.

We prove our claim with a double induction on $(r,\tilde{r})$. 

If $r=1$, $\mathcal{G}$ is a union of disconnected $0$-edges and applying our gluing construction amounts to connecting a $(-1)$-edge to $B$ in $\tilde{\mathcal{G}}$ then, eventually, taking a $(-1)$-sesqui-extension of the group represented by the result. By Lemma \ref{basecase1} and Proposition \ref{extremesesqui}, this yields a new string C-group.
If $\tilde{r}=1$, the same holds by symmetry.
The following table shows which cases are thus dealt with by the previous argument.

\begin{center}
\renewcommand{\arraystretch}{1.5}
$\begin{NiceArray}{w{c}{0.75cm}|c c c c}
     \diagbox{r}{\tilde{r\;}} & 1 & 2 & 3 & ...
     \\ \hline
     1 & (1,1) & (1,2) & (1,3) & ...
     \\
     2 & (2,1) & (2,2) & (2,3) & ...
     \\
     3 & (3,1) & (3,2) & (3,3) & ...
     \\
     4 & (4,1) & (4,2) & (4,3) & ...
     \\
     \vdots & \vdots & \vdots & \vdots & \ddots
\CodeAfter
\begin{tikzpicture}
\node (1) at (2-2) {\Large{\color{green}\ding{52}}};
\node (2) at (2-3) {\Large{\color{green}\ding{52}}};
\node (3) at (2-4) {\Large{\color{green}\ding{52}}};
\node (3) at (3-2) {\Large{\color{green}\ding{52}}};
\node (5) at (4-2) {\Large{\color{green}\ding{52}}};
\node (6) at (5-2) {\Large{\color{green}\ding{52}}};

\node (7) at (2-5) {\Large{\color{green}\ding{52}}};
\node (8) at (6-2) {\Large{\color{green}\ding{52}}};
\end{tikzpicture}
\end{NiceArray}$

\end{center}

Now suppose that the result holds when applying our gluing construction to appropriately shaped permutation representation graphs representing string C-groups of any ranks $R_1$ and $R_2$ such that $R_1 \leq r$ and $R_2 \leq \tilde{r}$ and $(R_1,R_2)\neq (r,\tilde{r})$.

Let $\Gamma$ be the permutation group whose generators $\rho_{-r},...,\rho_{-1},\rho_0,...,\rho_{\tilde{r}-1}$ are represented by $\mathcal{H}$ (so that $\rho_l=\tilde{\rho}_{-(l+1)}$ for all $l\in\{-r,...,-1\}$).
By Proposition \ref{breakinparts}, it is sufficient to show that $\Gamma_{-r}$ and $\Gamma_{\tilde{r}-1}$ are string C-groups and that $\Gamma_{-r}\cap\Gamma_{\tilde{r}-1}=\Gamma_{-r,\tilde{r}-1}$.

The permutation representation graphs of $\Gamma_{-r}$ and $\Gamma_{\tilde{r}-1}$ result respectively from applying our gluing construction to the CPR graphs of $G_{r-1}$ and $\tilde{G}$ and to the CPR graphs of $G$ and $\tilde{G}_{\tilde{r}-1}$. Since $G_{r-1}$ and $\tilde{G}$ are of ranks $r-1$ and $\tilde{r}$ respectively,  $\Gamma_{-r}$ is a CPR graph by induction hypothesis. Similarly, $\Gamma_{\tilde{r}-1}$ is also a CPR graph by induction hypothesis.

We are left with checking that $\Gamma_{-r}\cap\Gamma_{\tilde{r}-1}=\Gamma_{-r,\tilde{r}-1}$.
It is always true that $\Gamma_{-r}\cap\Gamma_{\tilde{r}-1}\geq\Gamma_{-r,\tilde{r}-1}$; and there is equality if $\Gamma_{-r}\cap\Gamma_{\tilde{r}-1}$ acts in the same way as $\Gamma_{-r,\tilde{r}-1}$ on the orbits of the latter.

We remind the reader that $C$ is the vertex at which $\mathcal G$ and $\mathcal G'$ are glued (see statement of the theorem).

The orbits of $\Gamma_{-r,\tilde{r}-1}$ as acting on the vertices of $\mathcal{H}$ are $\mathcal{O}:=C^{\Gamma_{-r,\tilde{r}-1}}$ and sets of vertices $\mathcal{O}_1,...,\mathcal{O}_{k_1},...,\mathcal{O}_{k_2}$ ($k_1\leq k_2\in\mathbb{N}$) connected either by edges of labels $-r+1,...,-2$ (for $\mathcal{O}_1,...,\mathcal{O}_{k_1}$) or by edges of labels $1,...,\tilde{r}-2$ (for $\mathcal{O}_{k_1+1},...,\mathcal{O}_{k_2}$).
Let $H$, respectively $K$, be the group induced by the action of $\Gamma_{-r,\tilde{r}-1}$ on $\mathcal{O}_1\cup...\cup\mathcal{O}_{k_1}$, respectively $\mathcal{O}_{k_1+1}\cup...\cup\mathcal{O}_{k_2}$. Since the edges connecting the vertices in $\mathcal{O}_1\cup...\cup\mathcal{O}_{k_1}$ have no common label with the edges connecting the vertices of $\mathcal{O}_{k_1+1}\cup...\cup\mathcal{O}_{k_2}$, we have that $\Gamma_{-r,\tilde{r}-1}$ acts as $H\times K$ on $\mathcal{O}_1\cup...\cup\mathcal{O}_{k_2}$.
By Lemma \ref{howact1}, $\Gamma_{-r,\tilde{r}-1}$ acts as $Sym(\mathcal O)$ on $\mathcal{O}$ and independently from the way it acts on $\mathcal{O}_1\cup...\cup\mathcal{O}_{k_2}$: $\Gamma_{-r,\tilde{r}-1}\cong S_t\times (H\times K)$ where $t=\vert\mathcal{O}\vert$.
Since none of the vertices of $\mathcal{O}_1\cup...\cup\mathcal{O}_{k_1}$ are connected by any $(\tilde{r}-1)$-edge, $\Gamma_{-r}$, and hence $\Gamma_{-r}\cap\Gamma_{\tilde{r}-1}$, acts in the same way as $\Gamma_{-r,\tilde{r}-1}$ on $\mathcal{O}_1\cup...\cup\mathcal{O}_{k_1}$.
Similarly, since none of the vertices of $\mathcal{O}_{k_1+1}\cup...\cup\mathcal{O}_{k_2}$ are connected by any $(-r)$-edge, $\Gamma_{\tilde{r}-1}$, and hence $\Gamma_{-r}\cap\Gamma_{\tilde{r}-1}$, acts in the same way as $\Gamma_{-r,\tilde{r}-1}$ on $\mathcal{O}_{k_1+1}\cup...\cup\mathcal{O}_{k_2}$.
By Lemma \ref{howact1}, both $\Gamma_{-r}$ and $\Gamma_{\tilde{r}-1}$ act as symmetric groups on their orbit containing $C$ (and hence on $\mathcal{O}$) and independently on their other orbits. Therefore $\Gamma_{-r}\cap\Gamma_{\tilde{r}-1}$ acts as $S_t\times (H\times K)$ on the orbits of $\Gamma_{-r,\tilde{r}-1}$ and the equality holds as required.
\end{proof}

\section{A not so successful gluing method}\label{doesntwork}
Looking at many examples of CPR graphs, it seemed natural to us to suggest the following conjecture.
\begin{conjecture}\label{false}
Let $i\geq 2$. Let $\mathcal{G}$ be a CPR graph of the following form

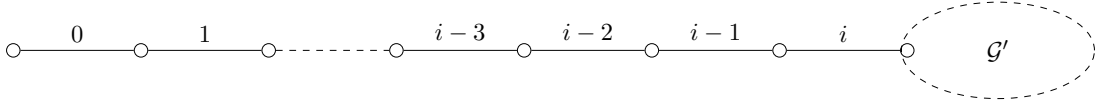
\begin{center}
\resizebox{\columnwidth}{!}{
\begin{tikzpicture}
\vertex[] (1) at (0,0) {};
\vertex[] (2) at (2,0) {};
\vertex[] (3) at (4,0) {};
\vertex[] (4) at (6,0) {};
\vertex[] (5) at (8,0) {};
\vertex[] (6) at (10,0) {};
\vertex[] (7) at (12,0) {};
\vertex[] (8) at (14,0) {};
\node[draw, right=-6pt of 8, shape=ellipse, style={dashed, minimum width=86pt, minimum height=43pt}] (9) {$\mathcal{G}'$};
\draw[-] (1) -- (2)
 node[pos=0.5,above] {$0$};
\draw[-] (3) -- (2)
 node[pos=0.5,above] {$1$};
\draw[dashed] (3) -- (4);
\draw[-] (5) -- (4)
 node[pos=0.5,above] {$i-3$};
\draw[-] (5) -- (6)
 node[pos=0.5,above] {$i-2$};
\draw[-] (7) -- (6)
 node[pos=0.5,above] {$i-1$};
\draw[-] (7) --(8)
 node[pos=0.5,above] {$i$};
\end{tikzpicture}}
\end{center}

where $\mathcal{G}'$ is a subgraph of $\mathcal{G}$.
Suppose that if $\mathcal{G}'$ is non-empty then all its edges have labels greater than or equal to $i$.
Then the following permutation representation graph is a CPR graph as well

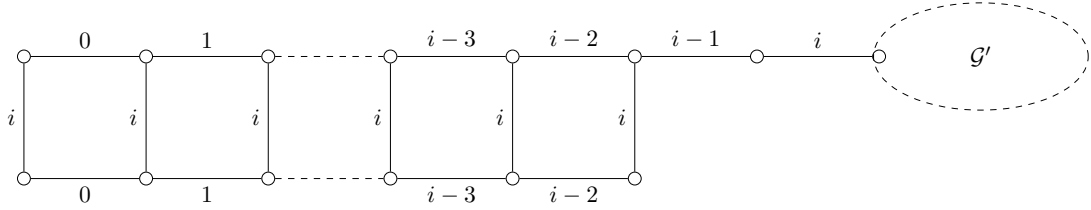
\begin{center}
\resizebox{\columnwidth}{!}{
\begin{tikzpicture}
\vertex[] (1) at (0,2) {};
\vertex[] (2) at (2,2) {};
\vertex[] (3) at (4,2) {};
\vertex[] (4) at (6,2) {};
\vertex[] (5) at (8,2) {};
\vertex[] (6) at (10,2) {};
\vertex[] (7) at (12,2) {};
\vertex[] (8) at (14,2) {};

\vertex[] (10) at (0,0) {};
\vertex[] (11) at (2,0) {};
\vertex[] (12) at (4,0) {};
\vertex[] (13) at (6,0) {};
\vertex[] (14) at (8,0) {};
\vertex[] (15) at (10,0) {};

\node[draw, right=-6pt of 8, shape=ellipse, style={dashed, minimum width=100pt, minimum height=50pt}] (9) {$\mathcal{G}'$};
\draw[-] (1) -- (2)
 node[pos=0.5,above] {$0$};
\draw[-] (3) -- (2)
 node[pos=0.5,above] {$1$};
\draw[dashed] (3) -- (4);
\draw[-] (5) -- (4)
 node[pos=0.5,above] {$i-3$};
\draw[-] (5) -- (6)
 node[pos=0.5,above] {$i-2$};
\draw[-] (7) -- (6)
 node[pos=0.5,above] {$i-1$};
\draw[-] (7) --(8)
 node[pos=0.5,above] {$i$};

\draw[-] (10) -- (11)
 node[pos=0.5,below] {$0$};
\draw[-] (12) -- (11)
 node[pos=0.5,below] {$1$};
\draw[dashed] (12) -- (13);
\draw[-] (14) -- (13)
 node[pos=0.5,below] {$i-3$};
\draw[-] (14) -- (15)
 node[pos=0.5,below] {$i-2$};

\draw[-] (1) -- (10)
 node[pos=0.5,left] {$i$};
\draw[-] (2) -- (11)
 node[pos=0.5,left] {$i$};
\draw[-] (3) -- (12)
 node[pos=0.5,left] {$i$};
\draw[-] (4) -- (13)
 node[pos=0.5,left] {$i$};
\draw[-] (5) -- (14)
 node[pos=0.5,left] {$i$};
\draw[-] (6) -- (15)
 node[pos=0.5,left] {$i$};

\end{tikzpicture}}
\end{center}

\end{conjecture}
This conjecture happens to be false. In fact, a graph put forward in \cite[Table 6]{woof} gives the perfect counterexample.

Before proving this, we give the following non-standard definition in order to simplify the reading of the remaining results of this chapter. We also state and prove a few useful lemmas.

For an integer $r\geq 1$, the {\em rank-$r$ simplex} (or {\em $r$-simplex}) is the permutation representation graph of the following shape.
 \begin{center}
 \resizebox{0.8\columnwidth}{!}{
        \begin{tikzpicture}
\vertex[] (1) at (0,1) {};
\vertex[] (2) at (2,1) {};
\vertex[] (3) at (4,1) {};
\vertex[] (4) at (6,1) {};
\vertex[] (5) at (8,1) {};
\vertex[] (6) at (10,1) {};
\vertex[] (7) at (12,1) {};
\draw[-] (1) -- (2)
 node[pos=0.5,above] {$0$};
 \draw[-] (3) -- (2)
 node[pos=0.5,above] {$1$};
 \draw[dashed] (3) -- (4);
 \draw[-] (4) -- (5)
 node[pos=0.5,above] {$r-3$};
 \draw[-] (5) -- (6)
 node[pos=0.5,above] {$r-2$};
 \draw[-] (6) -- (7)
 node[pos=0.5,above] {$r-1$};
 \end{tikzpicture}}
 \end{center}

We have the following well-known result.

\begin{lemma}\label{simplices}
    For an integer $r\geq 1$, a rank-$r$ simplex is a CPR graph representing $S_{r+1}$.
\end{lemma}

We often also use the name ``$r$-simplex'' to designate the group represented by an $r$-simplex.

\begin{lemma}\label{multisimp}
  For any integer $r\geq 1$, a union of identical rank-$r$ simplexes is a CPR graph representing $S_{r+1}$.
\end{lemma}
\begin{proof}
 Let $\Gamma$ be the group represented by such a union $\mathcal{G}$ of $k$ identical rank-$r$ simplexes. Numbering the vertices of $\mathcal{G}$ simplex after simplex, we have $\Gamma=\langle \prod_{j=0}^{k} (i+j(r+1),i+1+j(r+1)) \vert i\in\{1,...,r\}\rangle$ and $$\phi:(i,i+1) \mapsto \prod_{j=0}^{k} (i+j(r+1),i+1+j(r+1))$$ gives a natural isomorphism between $S_{r+1}$ and $\Gamma$.
\end{proof}

\begin{lemma}\cite[Lemma 6.1]{RankofPolAltGroups}\label{result1}
   For any integers $r\geq 2$ and $1\leq h\leq r-1$, the following permutation representation graph $\mathcal{G}$ is a CPR graph for $S_{h+1}\times S_{r+1}$.

    \begin{center}
    \resizebox{0.8\columnwidth}{!}{
\begin{tikzpicture}
\vertex[] (1) at (0,1) {};
\vertex[] (2) at (2,1) {};
\vertex[] (3) at (4,1) {};
\vertex[] (4) at (6,1) {};

\vertex[] (5) at (0,0) {};
\vertex[] (6) at (2,0) {};
\vertex[] (7) at (4,0) {};
\vertex[] (8) at (6,0) {};
\vertex[] (9) at (8,0) {};
\vertex[] (10) at (10,0) {};
\vertex[] (11) at (12,0) {};

\draw[-] (1) -- (2)
 node[pos=0.5,above] {$0$};
\draw[dashed] (3) -- (2);
\draw[-] (3) -- (4)
 node[pos=0.5,above] {$h-1$};

\draw[-] (5) -- (6)
 node[pos=0.5,below] {$0$};
 \draw[dashed] (6) -- (7);
\draw[-] (7) -- (8)
 node[pos=0.5,below] {$h-1$};
\draw[-] (9) -- (8)
 node[pos=0.5,below] {$h$};
  \draw[dashed] (9) -- (10);
\draw[-] (10) --(11)
 node[pos=0.5,below] {$r-1$};

\end{tikzpicture}}
\end{center}
    
\end{lemma}

Note that it is not, in general, true to say that any union of simplexes is a CPR graph. Consider, for example, the following union of simplexes representing a group $\Gamma$.

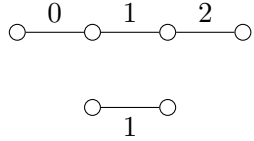
\begin{center}
\begin{tikzpicture}
    \vertex[] (1) at (0,1) {};
    \vertex[] (2) at (1,1) {};
    \vertex[] (3) at (2,1) {};
    \vertex[] (4) at (3,1) {};

    \vertex[] (5) at (1,0) {};
    \vertex[] (6) at (2,0) {};

    \draw[-] (1) -- (2)
     node[pos=0.5,above] {$0$};
    \draw[-] (3) -- (2)
     node[pos=0.5,above] {$1$};
    \draw[-] (3) -- (4)
     node[pos=0.5,above] {$2$};

    \draw[-] (5) -- (6)
     node[pos=0.5,below] {$1$};
\end{tikzpicture}
\end{center}

By Lemma \ref{result1}, $\Gamma_0$ and $\Gamma_2$ are isomorphic to $S_3\times C_2$ so that $C_2\times C_2\cong \Gamma_0\cap\Gamma_2\neq \Gamma_{0,2} \cong C_2$.

\begin{lemma}\label{wreathsimp}
    For any integer $r\geq 2$, the following permutation representation graph $\mathcal{G}^{(r)}$ is a CPR graph for $C_2 \wr S_r$.

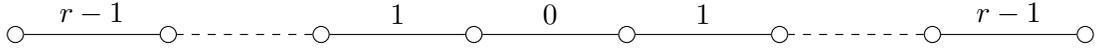
    \begin{center}
    \resizebox{\columnwidth}{!}{
\begin{tikzpicture}
\vertex[] (1) at (0,2) {};
\vertex[] (2) at (2,2) {};
\vertex[] (3) at (4,2) {};
\vertex[] (4) at (6,2) {};
\vertex[] (5) at (8,2) {};
\vertex[] (6) at (10,2) {};
\vertex[] (7) at (12,2) {};
\vertex[] (8) at (14,2) {};

\draw[-] (1) -- (2)
 node[pos=0.5,above] {$r-1$};
\draw[dashed] (3) -- (2);
\draw[-] (3) -- (4)
 node[pos=0.5,above] {$1$};
\draw[-] (5) -- (4)
 node[pos=0.5,above] {$0$};
\draw[-] (5) -- (6)
 node[pos=0.5,above] {$1$};
\draw[dashed] (7) -- (6);
\draw[-] (7) --(8)
 node[pos=0.5,above] {$r-1$};
\end{tikzpicture}}
\end{center}

\end{lemma}
\begin{proof}
 Let $\Gamma^{(r)}=\langle \rho_0,\rho_1,...,\rho_{r-1}\rangle$ be the group represented by $\mathcal{G}^{(r)}$.
 First note that $\Gamma^{(r)}$ acts imprimitively on the vertex set of $\mathcal{G}^{(r)}$ with $r$ blocks of size two $\mathcal{B}_i=\{a_i,b_i\}$ ($i\in\{1,...,r\}$) as illustrated below.

 \begin{center}
    \resizebox{\columnwidth}{!}
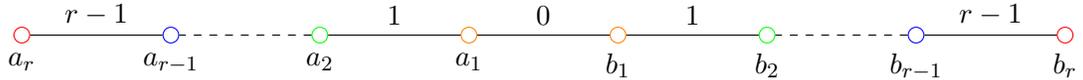
{
\begin{tikzpicture}
\vertex[color=red,label=below:$a_r$] (1) at (0,2) {};
\vertex[color=blue,label=below:$a_{r-1}$] (2) at (2,2) {};
\vertex[color=green,label=below:$a_2$] (3) at (4,2) {};
\vertex[color=orange,label=below:$a_1$] (4) at (6,2) {};
\vertex[color=orange,label=below:$b_1$] (5) at (8,2) {};
\vertex[color=green,label=below:$b_2$] (6) at (10,2) {};
\vertex[color=blue,label=below:$b_{r-1}$] (7) at (12,2) {};
\vertex[color=red,label=below:$b_r$] (8) at (14,2) {};

\draw[-] (1) -- (2)
 node[pos=0.5,above] {$r-1$};
\draw[dashed] (3) -- (2);
\draw[-] (3) -- (4)
 node[pos=0.5,above] {$1$};
\draw[-] (5) -- (4)
 node[pos=0.5,above] {$0$};
\draw[-] (5) -- (6)
 node[pos=0.5,above] {$1$};
\draw[dashed] (7) -- (6);
\draw[-] (7) --(8)
 node[pos=0.5,above] {$r-1$};
\end{tikzpicture}}
\end{center}

By Theorem \ref{wreath}, $\Gamma^{(r)}\leq C_2\wr S_r$. By Lemma \ref{simplices}, $\Gamma^{(r)}$ acts as $S_r$ on $\{\mathcal{B}_1,...,\mathcal{B}_r\}$. Moreover, $\rho_0=(a_1,b_1)$ and, for any $j\in\{2,...,r\}$, $\rho_0^{\prod_{i=1}^{j-1} \rho_i}=(a_j,b_j)$. Hence $\Gamma^{(r)}$ is isomorphic to the full wreath product $C_2\wr S_r$.
As it is clear that $\Gamma^{(r)}$ satisfies the string property, we are left with showing that it satisfies the intersection property.
We do so by induction on $r\geq 2$.

For $r=2$, $\Gamma^{(r)}$ is dihedral and hence clearly a string C-group.

Suppose that $\Gamma^{(r-1)}$ is a string C-group for some $r\geq 3$.
By induction hypothesis, $\Gamma^{(r)}_{r-1}$ is a string C-group. By Lemma \ref{multisimp}, $\Gamma^{(r)}_0$ is also string C-group, isomorphic to $S_r$. Looking at the action of $\Gamma^{(r)}_0\cap\Gamma^{(r)}_{r-1}\leq \Gamma^{(r)}_0$ on the orbits of $\Gamma^{(r)}_{0,r-1}$, we readily see that $\Gamma^{(r)}_0\cap\Gamma^{(r)}_{r-1}\leq S_{r-1}\cong \Gamma^{(r)}_{0,r-1}$ and thus $\Gamma^{(r)}_0\cap\Gamma^{(r)}_{r-1}=\Gamma^{(r)}_{0,r-1}$. By Proposition \ref{breakinparts}, $\Gamma^{(r)}$ is a string C-group and, by induction, this is true for all $r\geq 2$.
\end{proof}

\begin{lemma}\label{lemme1}
    For any integer $r\geq 3$, the following permutation representation graph $\mathcal{G}^{(r)}$ is a CPR graph for $S_{r+2}\times S_r$.

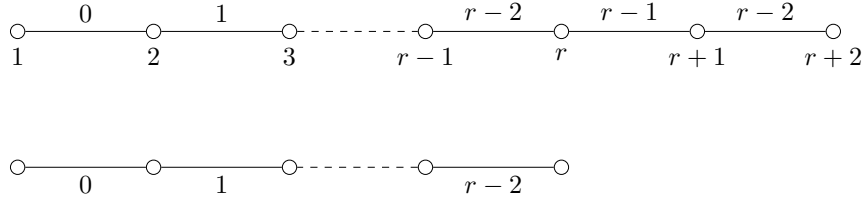
\begin{center}
\resizebox{0.8\columnwidth}{!}{
\begin{tikzpicture}
\vertex[label=below:$1$] (1) at (0,2) {};
\vertex[label=below:$2$] (2) at (2,2) {};
\vertex[label=below:$3$] (3) at (4,2) {};
\vertex[label=below:$r-1$] (4) at (6,2) {};
\vertex[label=below:$r$] (5) at (8,2) {};
\vertex[label=below:$r+1$] (6) at (10,2) {};
\vertex[label=below:$r+2$] (7) at (12,2) {};

\vertex[] (8) at (0,0) {};
\vertex[] (9) at (2,0) {};
\vertex[] (10) at (4,0) {};
\vertex[] (11) at (6,0) {};
\vertex[] (12) at (8,0) {};

\draw[-] (1) -- (2)
 node[pos=0.5,above] {$0$};
\draw[-] (2) -- (3)
 node[pos=0.5,above] {$1$};
\draw[dashed] (3) -- (4);
\draw[-] (5) -- (4)
 node[pos=0.5,above] {$r-2$};
\draw[-] (5) -- (6)
 node[pos=0.5,above] {$r-1$};
\draw[-] (7) -- (6)
 node[pos=0.5,above] {$r-2$}; 

\draw[-] (9) --(8)
 node[pos=0.5,below] {$0$};
\draw[-] (10) --(9)
 node[pos=0.5,below] {$1$};
\draw[dashed] (11) -- (10);
\draw[-] (12) --(11)
 node[pos=0.5,below] {$r-2$};
\end{tikzpicture}}
\end{center}
    
\end{lemma}
\begin{proof}
We first show that, for any $r\geq 3$, the permutation group $\Gamma^{(r)}=\langle\rho_0,\rho_1,...,\rho_{r-1}\rangle$ represented by $\mathcal{G}^{(r)}$ is isomorphic to $S_{r+2}\times S_r$.

Let $\mathcal{O}^{(r)}_1$ and $\mathcal{O}^{(r)}_2$ be the vertex sets of the largest, respectively smallest, connected components of $\mathcal{G}^{(r)}$.
The group acting on the $r+2$ vertices of $\mathcal{O}^{(r)}_1$ is isomorphic to $S_{r+2}$ by~\cite[Theorem 2]{CPRSn}.
By Lemma \ref{simplices}, $\Gamma^{(r)}$ acts as $S_r$ on $\mathcal{O}^{(r)}_2$.
Finally, since $\Gamma^{(r)}$ acts as $S_{r+2}$ on $\mathcal{O}^{(r)}_1$, for any $i\in\{1,...,r+1\}$, there is $\sigma\in\Gamma^{(r)}$ such that $(i,i+1)=\rho_{r-1}^\sigma$. These elements all fix $\mathcal{O}^{(r)}_2$ pointwise so $\Gamma^{(r)}\cong S_{r+2}\times S_r$.

To show that $\mathcal{G}^{(r)}$ is a CPR graph or, equivalently, that $\Gamma^{(r)}$ is a string C-group, we first note that $\Gamma^{(r)}$ clearly satisfies the string property and show, by induction on $r$, that it also satisfies the intersection property.

One easily checks that $\Gamma^{(r)}$ is a string C-group when $r=3$, by hand or using {\sc Magma}~\cite{BCP97}.
Suppose that $\Gamma^{(r-1)}$ is a string C-group for some $r\geq 4$.
Then $\Gamma^{(r)}_0$ is a string C-group since it is isomorphic to $\Gamma^{(r-1)}$.
But $\Gamma^{(r)}_{r-1}$ is also a string C-group by Proposition \ref{extremesesqui} since it is an $(r-2)$-sesqui-extension of a rank-$(r-1)$ group that is a string C-group by Lemma \ref{multisimp}.
An argument fairly similar to the one given above to show that $\Gamma^{(r)}\cong S_{r+2}\times S_r$ yields $\Gamma^{(r)}_{0,r-1}\cong S_{r-1}\times C_2$ and $\Gamma^{(r)}_{r-1}\cong S_r\times C_2$ so that $\Gamma^{(r)}_0\cap\Gamma^{(r)}_{r-1}=\Gamma^{(r)}_{0,r-1}$. By Proposition \ref{breakinparts}, $\Gamma^{(r)}$ is a string C-group and thus, by induction, this is true for all $r\geq 3$.
\end{proof}

    Note that $\mathcal{G}^{(r)}$ as defined in Lemma \ref{lemme1} is also a CPR graph for $r=2$ since it then represents a dihedral group but, in this case, it represents $C_2\wr C_2\cong D_8$.
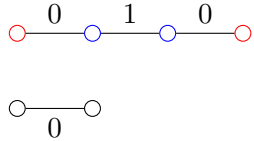
\begin{center}
\begin{tikzpicture}
    \vertex[color=red] (1) at (0,1) {};
    \vertex[color=blue] (2) at (1,1) {};
    \vertex[color=blue] (3) at (2,1) {};
    \vertex[color=red] (4) at (3,1) {};
    \vertex[] (5) at (0,0) {};
    \vertex[] (6) at (1,0) {};

    \draw[-] (1) -- (2)
    node[pos=0.5,above] {$0$};
    \draw[-] (3) -- (2)
    node[pos=0.5,above] {$1$};
    \draw[-] (3) -- (4)
    node[pos=0.5,above] {$0$};
    \draw[-] (5) -- (6)
    node[pos=0.5,below] {$0$};
\end{tikzpicture}
\end{center}

With these lemmas, we can now give a counterexample to Conjecture \ref{false}. We do so in two propositions: the first gives a CPR graph and the second establishes that the gluing method described in Conjecture \ref{false} is not successful when applied to this CPR graph.

Observe that the permutation representation graph considered in Proposition \ref{counterexample2} is graph (X) of~\cite[Table 6]{woof} with $h\leq r-4$ thus Proposition \ref{counterexample2} already disproves the conjecture of Cameron, Fernandes and Leemans. 

\begin{proposition}\label{counterexample1}
    For any integers $r\geq 3$ and $1\leq h\leq r-2$, the following permutation representation graph $\mathcal{G}$ is a CPR graph for $S_{r+h+1}$.
    \begin{center}
    \resizebox{\columnwidth}{!}{
\begin{tikzpicture}
\vertex[] (1) at (0,2) {};
\vertex[] (2) at (2,2) {};
\vertex[] (3) at (4,2) {};
\vertex[] (4) at (6,2) {};
\vertex[] (5) at (8,2) {};
\vertex[] (6) at (10,2) {};
\vertex[] (7) at (12,2) {};
\vertex[] (8) at (14,2) {};
\vertex[] (9) at (16,2) {};
\vertex[] (10) at (18,2) {};
\vertex[] (11) at (20,2) {};

\draw[-] (1) -- (2)
 node[pos=0.5,above] {$h$};
\draw[dashed] (3) -- (2);
\draw[-] (3) -- (4)
 node[pos=0.5,above] {0};
\draw[dashed] (5) -- (4);
\draw[-] (5) -- (6)
 node[pos=0.5,above] {$h$};
\draw[-] (7) -- (6)
 node[pos=0.5,above] {$h+1$};
\draw[-] (7) -- (8)
 node[pos=0.5,above] {$h+2$};
\draw[-] (9) --(8)
 node[pos=0.5,above] {$h+3$};
\draw[dashed] (9) -- (10);
\draw[-] (10) --(11)
 node[pos=0.5,above] {$r-1$};
\end{tikzpicture}}
\end{center}
\end{proposition}
\begin{proof}
We first prove that, for any $h\in\mathbb{N}_0$, the following permutation representation graph $\mathcal{G}''$ is a CPR graph for $S_{2h+3}$.

    \begin{center}
    \resizebox{0.8\columnwidth}{!}
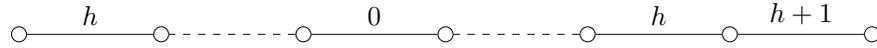
{
\begin{tikzpicture}
\vertex[] (1) at (0,2) {};
\vertex[] (2) at (2,2) {};
\vertex[] (3) at (4,2) {};
\vertex[] (4) at (6,2) {};
\vertex[] (5) at (8,2) {};
\vertex[] (6) at (10,2) {};
\vertex[] (7) at (12,2) {};

\draw[-] (1) -- (2)
 node[pos=0.5,above] {$h$};
\draw[dashed] (3) -- (2);
\draw[-] (3) -- (4)
 node[pos=0.5,above] {0};
\draw[dashed] (5) -- (4);
\draw[-] (5) -- (6)
 node[pos=0.5,above] {$h$};
\draw[-] (7) -- (6)
 node[pos=0.5,above] {$h+1$};
\end{tikzpicture}}
\end{center}

This graph corresponds to $\mathcal{G}$ when $r=h+2$, to a subgraph of $\mathcal{G}$ otherwise. 
Let $\Gamma''$ be the group represented by $\mathcal{G}''$.
By Lemma \ref{result1}, $\Gamma''_0$ is a string C-group isomorphic to $S_{h+1}\times S_{h+2}$. By Lemma \ref{wreathsimp}, $\Gamma''_{h+1}$ is a string C-group isomorphic to $C_2 \wr S_{h+1}$ and, by Lemma \ref{multisimp},  $\Gamma''_{0,h+1}\cong S_{h+1}$. Looking at the actions of $\Gamma''_0$ and $\Gamma''_{h+1}$ on the orbits of $\Gamma''_{0,h+1}$, we readily see that $\Gamma''_0\cap\Gamma''_{h+1}=\Gamma''_{0,h+1}$.
By Proposition \ref{breakinparts}, $\Gamma''$ is a string C-group and, equivalently, $\mathcal{G}''$ is a CPR graph.
Since $\Gamma''_0\cong S_{h+1}\times S_{h+2}$ and $S_{h+1}\times S_{h+2}$ is maximal in $S_{2h+3}$, $\Gamma''\cong S_{2h+3}$.

When $r>h+2$, repeated applications of the edge-addition procedure described in Lemma \ref{basecase1} to the dual of $\mathcal{G}''$ followed by a final dualization yield $\mathcal{G}$. By Lemma \ref{basecase1}, since $\mathcal{G}$ is connected and has $r+h+1$ vertices, it is a CPR graph and represents $S_{r+h+1}$, as claimed.
\end{proof}


\begin{proposition}\label{counterexample2}
For any integers $r\geq 5$ and $1\leq h\leq r-4$, the following permutation representation graph $\mathcal{G}'$ is not a CPR graph (though it does represent $S_{2r-1}$).
    \begin{center}
    \resizebox{\columnwidth}{!}
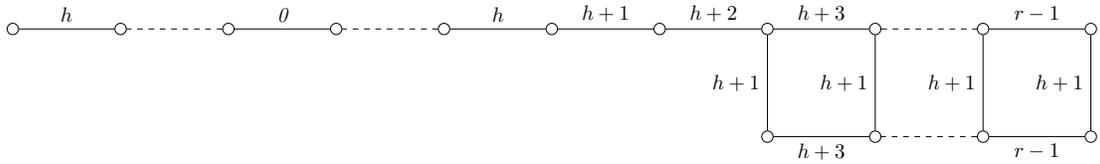
{
\begin{tikzpicture}
\vertex[] (1) at (0,2) {};
\vertex[] (2) at (2,2) {};
\vertex[] (3) at (4,2) {};
\vertex[] (4) at (6,2) {};
\vertex[] (5) at (8,2) {};
\vertex[] (6) at (10,2) {};
\vertex[] (7) at (12,2) {};
\vertex[] (8) at (14,2) {};
\vertex[] (9) at (16,2) {};
\vertex[] (10) at (18,2) {};
\vertex[] (11) at (20,2) {};

\draw[-] (1) -- (2)
 node[pos=0.5,above] {$h$};
\draw[dashed] (3) -- (2);
\draw[-] (3) -- (4)
 node[pos=0.5,above] {0};
\draw[dashed] (5) -- (4);
\draw[-] (5) -- (6)
 node[pos=0.5,above] {$h$};
\draw[-] (7) -- (6)
 node[pos=0.5,above] {$h+1$};
\draw[-] (7) -- (8)
 node[pos=0.5,above] {$h+2$};
\draw[-] (9) --(8)
 node[pos=0.5,above] {$h+3$};
\draw[dashed] (9) -- (10);
\draw[-] (10) --(11)
 node[pos=0.5,above] {$r-1$};

\vertex[] (12) at (14,0) {};
\vertex[] (13) at (16,0) {};
\vertex[] (14) at (18,0) {};
\vertex[] (15) at (20,0) {};

\draw[-] (12) -- (13)
 node[pos=0.5,below] {$h+3$};
\draw[dashed] (14) -- (13);
\draw[-] (14) -- (15)
 node[pos=0.5,below] {$r-1$};

\draw[-] (8) -- (12)
 node[pos=0.5,left] {$h+1$};
\draw[-] (9) -- (13)
 node[pos=0.5,left] {$h+1$};
\draw[-] (10) -- (14)
 node[pos=0.5,left] {$h+1$};
\draw[-] (11) -- (15)
 node[pos=0.5,left] {$h+1$};
 
\end{tikzpicture}}
\end{center}
\end{proposition}

\begin{proof}
Let $\Gamma'=\langle\rho_0,...,\rho_h,...,\rho_{r-1}\rangle$ be the group represented by $\mathcal G'$.

We show that $\Gamma'_{0,h+2\to r-1}\neq\Gamma'_{0,h+3\to r-1}\cap\Gamma'_{h+2\to r-1}$.

Let $\sigma:=(a,b)\rho_{h+1}=\prod_{i=1}^{r-h-2} (i,(r-h-2)+i)$ with the vertices of $\mathcal{G}'$ numbered or labelled as below.

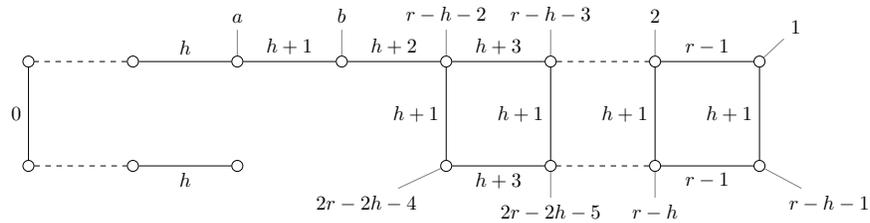
\begin{center}
\resizebox{0.8\columnwidth}{!}{
\begin{tikzpicture}
\vertex[] (1) at (10,0) {};
\vertex[] (2) at (8,0) {};
\vertex[] (3) at (6,0) {};
\vertex[] (4) at (6,2) {};
\vertex[] (5) at (8,2) {};
\vertex[pin=above:$a$] (6) at (10,2) {};
\vertex[pin=above:$b$] (7) at (12,2) {};
\vertex[pin=above:$r-h-2$] (8) at (14,2) {};
\vertex[pin=above:$r-h-3$] (9) at (16,2) {};
\vertex[pin=above:$2$] (10) at (18,2) {};
\vertex[pin=40:$1$] (11) at (20,2) {};

\draw[-] (1) -- (2)
 node[pos=0.5,below] {$h$};
\draw[dashed] (3) -- (2);
\draw[-] (3) -- (4)
 node[pos=0.5,left] {$0$};
\draw[dashed] (5) -- (4);
\draw[-] (5) -- (6)
 node[pos=0.5,above] {$h$};
\draw[-] (7) -- (6)
 node[pos=0.5,above] {$h+1$};
\draw[-] (7) -- (8)
 node[pos=0.5,above] {$h+2$};
\draw[-] (9) --(8)
 node[pos=0.5,above] {$h+3$};
\draw[dashed] (9) -- (10);
\draw[-] (10) --(11)
 node[pos=0.5,above] {$r-1$};

\vertex[pin=-135:$2r-2h-4$] (12) at (14,0) {};
\vertex[pin=below:$2r-2h-5$] (13) at (16,0) {};
\vertex[pin=below:$r-h$] (14) at (18,0) {};
\vertex[pin=-45:$r-h-1$] (15) at (20,0) {};

\draw[-] (12) -- (13)
 node[pos=0.5,below] {$h+3$};
\draw[dashed] (14) -- (13);
\draw[-] (14) -- (15)
 node[pos=0.5,below] {$r-1$};

\draw[-] (8) -- (12)
 node[pos=0.5,left] {$h+1$};
\draw[-] (9) -- (13)
 node[pos=0.5,left] {$h+1$};
\draw[-] (10) -- (14)
 node[pos=0.5,left] {$h+1$};
\draw[-] (11) -- (15)
 node[pos=0.5,left] {$h+1$};
 
\end{tikzpicture}}
\end{center}

We show that  $\sigma\notin\Gamma'_{0,h+2\to r-1}$ while $\sigma\in\Gamma'_{0,h+3\to r-1}$ and $\sigma\in\Gamma'_{h+2\to r-1}$.
Let $G=\langle \psi_1,...,\psi_h,\psi_{h+1}\rangle$  be the permutation group represented by the following permutation representation graph.

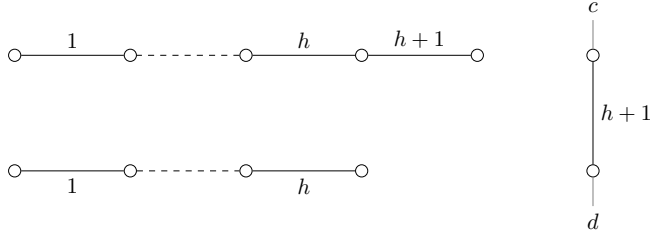
\begin{center}
\resizebox{0.6\columnwidth}{!}{
\begin{tikzpicture}
\vertex[] (1) at (10,0) {};
\vertex[] (2) at (8,0) {};
\vertex[] (3) at (6,0) {};
\vertex[] (4) at (6,2) {};
\vertex[] (5) at (8,2) {};
\vertex[] (6) at (10,2) {};
\vertex[] (7) at (12,2) {};
\vertex[pin=above:$c$] (8) at (14,2) {};

\vertex[] (9) at (4,0) {};
\vertex[] (10) at (4,2) {};
\draw[-] (3) -- (9)
 node[pos=0.5,below] {$1$};
\draw[-] (4) -- (10)
 node[pos=0.5,above] {$1$};

\draw[-] (1) -- (2)
 node[pos=0.5,below] {$h$};
\draw[dashed] (3) -- (2);

\draw[dashed] (5) -- (4);
\draw[-] (5) -- (6)
 node[pos=0.5,above] {$h$};
\draw[-] (7) -- (6)
 node[pos=0.5,above] {$h+1$};

\vertex[pin=below:$d$] (12) at (14,0) {};

\draw[-] (8) -- (12)
 node[pos=0.5,right] {$h+1$};
\end{tikzpicture}}
\end{center}

The homomorphism $\phi:\langle\rho_1,...,\rho_h,(a,b),\sigma\rangle\to\langle\psi_1,...,\psi_h,(c,d)\psi_{h+1},(c,d)\rangle$ that maps $\rho_i$ to $\psi_i$ for all $i\in\{1,...,h\}$, $(a,b)$ to $(c,d)\psi_{h+1}$ and $\sigma$ to $(c,d)$ yields an isomorphism between $\Gamma_{0,h+2\to r-1}$ and $G$.
Since all generators of $G$ are even, $(c,d)\notin G$ and thus $\sigma\notin \Gamma'_{0,h+2\to r-1}$.
By Proposition \ref{counterexample1}, $\Gamma'_{h+2\to r-1}$ acts on its largest orbit as $S_{2h+3}$. Hence there exists $\tau_1\in\Gamma'_{h+2\to r-1}$ such that $(a,b)=\rho_0^{\tau_1}$. Therefore, $\sigma=\rho_{h+1}\rho_0^{\tau_1}\in\Gamma'_{h+2\to r-1}$.

A similar argument, appealing to Lemma \ref{lemme1} this time, gives that there is $\tau_2\in\Gamma'_{0, h+3\to r-1}$ such that $(a,b)=\rho_{h+2}^{\tau_2}$ and thus $\sigma=\rho_{h+1}\rho_{h+2}^{\tau_2}\in\Gamma'_{0,h+3\to r-1}$.
\end{proof}

Together, Propositions \ref{counterexample1} and \ref{counterexample2} establish that Conjecture \ref{false} is false but it is nonetheless true if one adds the hypothesis that $\mathcal{G}$ is a simplex. Let us prove this, starting once again with a lemma.

\begin{lemma}\label{speccase}
For any integer $r\geq 3$,
    \begin{center}
\resizebox{0.8\columnwidth}{!}
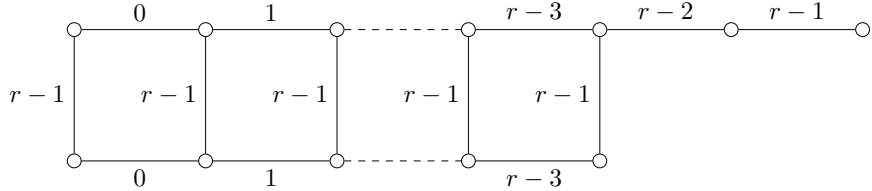
{
    \begin{tikzpicture}
\vertex[] (1) at (0,2) {};
\vertex[] (2) at (2,2) {};
\vertex[] (3) at (4,2) {};
\vertex[] (4) at (6,2) {};
\vertex[] (5) at (8,2) {};
\vertex[] (6) at (10,2) {};
\vertex[] (7) at (12,2) {};

\vertex[] (8) at (0,0) {};
\vertex[] (9) at (2,0) {};
\vertex[] (10) at (4,0) {};
\vertex[] (11) at (6,0) {};
\vertex[] (12) at (8,0) {};

\draw[-] (1) -- (2)
 node[pos=0.5,above] {$0$};
\draw[-] (3) -- (2)
 node[pos=0.5,above] {$1$};
\draw[dashed] (3) -- (4);
\draw[-] (5) -- (4)
 node[pos=0.5,above] {$r-3$};
\draw[-] (6) -- (5)
 node[pos=0.5,above] {$r-2$};
\draw[-] (7) -- (6)
 node[pos=0.5,above] {$r-1$};
 
\draw[-] (9) -- (8)
 node[pos=0.5,below] {$0$}; 
\draw[-] (9) -- (10)
 node[pos=0.5,below] {$1$}; 
\draw[dashed] (10) -- (11);
\draw[-] (11) -- (12)
 node[pos=0.5,below] {$r-3$};

\draw[-] (1) -- (8)
 node[pos=0.5,left] {$r-1$};
\draw[-] (2) -- (9)
 node[pos=0.5,left] {$r-1$};
\draw[-] (3) -- (10)
 node[pos=0.5,left] {$r-1$};
\draw[-] (4) -- (11)
 node[pos=0.5,left] {$r-1$};
\draw[-] (5) -- (12)
 node[pos=0.5,left] {$r-1$};
\end{tikzpicture}}
    \end{center}
    is a CPR graph for $S_r\wr C_2$.
\end{lemma}
\begin{proof}
Let us prove this result by induction on $r\geq 3$.
    
    For $r=3$, we have the following permutation representation graph.
    
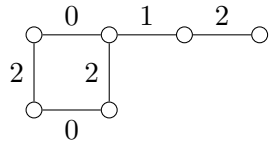
\begin{center}
    \begin{tikzpicture}
\vertex[] (1) at (0,1) {};
\vertex[] (2) at (1,1) {};
\vertex[] (3) at (2,1) {};
\vertex[] (4) at (3,1) {};

\vertex[] (5) at (0,0) {};
\vertex[] (6) at (1,0) {};

\draw[-] (1) -- (2)
 node[pos=0.5,above] {$0$};
\draw[-] (3) -- (2)
 node[pos=0.5,above] {$1$};
\draw[-] (3) -- (4)
 node[pos=0.5,above] {$2$};

\draw[-] (5) -- (6)
 node[pos=0.5,below] {$0$}; 

\draw[-] (1) -- (5)
 node[pos=0.5,left] {$2$};
\draw[-] (2) -- (6)
 node[pos=0.5,left] {$2$};

    \end{tikzpicture}
    \end{center}

Using {\sc Magma}, one can quickly see that the permutation group 
$$\langle (1,2)(5,6),(2,3),(3,4)(1,5)(2,6)\rangle$$ 
is a string C-group isomorphic to $S_3\wr C_2$.

Now suppose that the following graph $\mathcal G$ is a CPR graph.
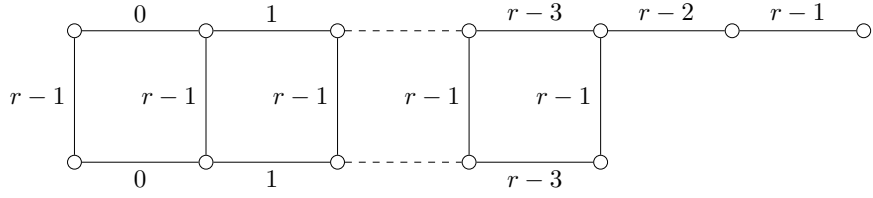
\begin{center}
\resizebox{0.8\columnwidth}{!}{
    \begin{tikzpicture}
\vertex[] (1) at (0,2) {};
\vertex[] (2) at (2,2) {};
\vertex[] (3) at (4,2) {};
\vertex[] (4) at (6,2) {};
\vertex[] (5) at (8,2) {};
\vertex[] (6) at (10,2) {};
\vertex[] (7) at (12,2) {};

\vertex[] (8) at (0,0) {};
\vertex[] (9) at (2,0) {};
\vertex[] (10) at (4,0) {};
\vertex[] (11) at (6,0) {};
\vertex[] (12) at (8,0) {};

\draw[-] (1) -- (2)
 node[pos=0.5,above] {$0$};
\draw[-] (3) -- (2)
 node[pos=0.5,above] {$1$};
\draw[dashed] (3) -- (4);
\draw[-] (5) -- (4)
 node[pos=0.5,above] {$r-3$};
\draw[-] (6) -- (5)
 node[pos=0.5,above] {$r-2$};
\draw[-] (7) -- (6)
 node[pos=0.5,above] {$r-1$};
 
\draw[-] (9) -- (8)
 node[pos=0.5,below] {$0$}; 
\draw[-] (9) -- (10)
 node[pos=0.5,below] {$1$}; 
\draw[dashed] (10) -- (11);
\draw[-] (11) -- (12)
 node[pos=0.5,below] {$r-3$};

\draw[-] (1) -- (8)
 node[pos=0.5,left] {$r-1$};
\draw[-] (2) -- (9)
 node[pos=0.5,left] {$r-1$};
\draw[-] (3) -- (10)
 node[pos=0.5,left] {$r-1$};
\draw[-] (4) -- (11)
 node[pos=0.5,left] {$r-1$};
\draw[-] (5) -- (12)
 node[pos=0.5,left] {$r-1$};
\end{tikzpicture}}
\end{center}

Let us show that the following graph is a CPR graph as well.

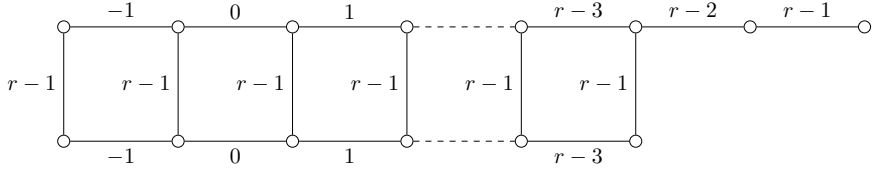
\begin{center}
\resizebox{0.8\columnwidth}{!}{
    \begin{tikzpicture}
\vertex[] (1) at (0,2) {};
\vertex[] (2) at (2,2) {};
\vertex[] (3) at (4,2) {};
\vertex[] (4) at (6,2) {};
\vertex[] (5) at (8,2) {};
\vertex[] (6) at (10,2) {};
\vertex[] (7) at (12,2) {};

\vertex[] (8) at (0,0) {};
\vertex[] (9) at (2,0) {};
\vertex[] (10) at (4,0) {};
\vertex[] (11) at (6,0) {};
\vertex[] (12) at (8,0) {};

\vertex[] (13) at (-2,2) {};
\vertex[] (14) at (-2,0) {};

\draw[-] (1) -- (2)
 node[pos=0.5,above] {$0$};
\draw[-] (3) -- (2)
 node[pos=0.5,above] {$1$};
\draw[dashed] (3) -- (4);
\draw[-] (5) -- (4)
 node[pos=0.5,above] {$r-3$};
\draw[-] (6) -- (5)
 node[pos=0.5,above] {$r-2$};
\draw[-] (7) -- (6)
 node[pos=0.5,above] {$r-1$};
 
\draw[-] (9) -- (8)
 node[pos=0.5,below] {$0$}; 
\draw[-] (9) -- (10)
 node[pos=0.5,below] {$1$}; 
\draw[dashed] (10) -- (11);
\draw[-] (11) -- (12)
 node[pos=0.5,below] {$r-3$};

\draw[-] (1) -- (8)
 node[pos=0.5,left] {$r-1$};
\draw[-] (2) -- (9)
 node[pos=0.5,left] {$r-1$};
\draw[-] (3) -- (10)
 node[pos=0.5,left] {$r-1$};
\draw[-] (4) -- (11)
 node[pos=0.5,left] {$r-1$};
\draw[-] (5) -- (12)
 node[pos=0.5,left] {$r-1$};

\draw[-] (1) -- (13)
 node[pos=0.5,above] {$-1$};
\draw[-] (8) -- (14)
 node[pos=0.5,below] {$-1$};
\draw[-] (14) -- (13)
 node[pos=0.5,left] {$r-1$};

\end{tikzpicture}}
    \end{center}
    
Let $\Gamma$ be the permutation group represented by this permutation representation graph.
First note that $\Gamma$ clearly satisfies the string property by Proposition \ref{cprlem}.
Now $\Gamma_{-1}$ is an $(r-1)$-sesqui-extension of the group represented by the graph $\mathcal G$ above and hence is a string C-group by Proposition \ref{extremesesqui}.
Moreover, $\Gamma_{r-1}$ is a string C-group by Lemma \ref{result1}.
Finally $\Gamma_{-1,r-1}$ is isomorphic to $S_{r+1}\times S_{r-1}$ by Lemma~\ref{result1}.
In $\Gamma_{r-1}$, all elements fix the set of the top vertices and the set of the bottom vertices. The pair of leftmost vertices is fixed by any element of $\Gamma_{-1}$.
Any element of $\Gamma_{-1}\cap\Gamma_{r-1}$ thus fixes both leftmost vertices as well as the set of the remaining top vertices and the set of the remaining bottom vertices. So $\Gamma_{-1}\cap\Gamma_{r-1}\leq \Gamma_{-1,r-1}\cong S_{r+1}\times S_{r-1}$ and hence $\Gamma_{-1}\cap\Gamma_{r-1}=\Gamma_{-1,r-1}$.
Since $\Gamma_{-1}$ and $\Gamma_{r-1}$ are string C-groups and since $\Gamma_{-1}\cap\Gamma_{r-1}=\Gamma_{-1,r-1}$, we have $\Gamma$ is a string C-group by Proposition \ref{breakinparts}.

It remains to show that the permutation group $G$ depicted by the permutation representation graph

    \begin{center}
\resizebox{0.8\columnwidth}{!}
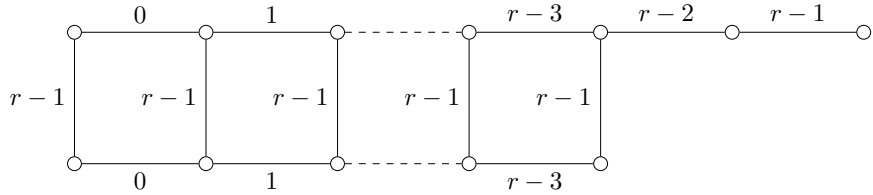
{
    \begin{tikzpicture}
\vertex[] (1) at (0,2) {};
\vertex[] (2) at (2,2) {};
\vertex[] (3) at (4,2) {};
\vertex[] (4) at (6,2) {};
\vertex[] (5) at (8,2) {};
\vertex[] (6) at (10,2) {};
\vertex[] (7) at (12,2) {};

\vertex[] (8) at (0,0) {};
\vertex[] (9) at (2,0) {};
\vertex[] (10) at (4,0) {};
\vertex[] (11) at (6,0) {};
\vertex[] (12) at (8,0) {};

\draw[-] (1) -- (2)
 node[pos=0.5,above] {$0$};
\draw[-] (3) -- (2)
 node[pos=0.5,above] {$1$};
\draw[dashed] (3) -- (4);
\draw[-] (5) -- (4)
 node[pos=0.5,above] {$r-3$};
\draw[-] (6) -- (5)
 node[pos=0.5,above] {$r-2$};
\draw[-] (7) -- (6)
 node[pos=0.5,above] {$r-1$};
 
\draw[-] (9) -- (8)
 node[pos=0.5,below] {$0$}; 
\draw[-] (9) -- (10)
 node[pos=0.5,below] {$1$}; 
\draw[dashed] (10) -- (11);
\draw[-] (11) -- (12)
 node[pos=0.5,below] {$r-3$};

\draw[-] (1) -- (8)
 node[pos=0.5,left] {$r-1$};
\draw[-] (2) -- (9)
 node[pos=0.5,left] {$r-1$};
\draw[-] (3) -- (10)
 node[pos=0.5,left] {$r-1$};
\draw[-] (4) -- (11)
 node[pos=0.5,left] {$r-1$};
\draw[-] (5) -- (12)
 node[pos=0.5,left] {$r-1$};
\end{tikzpicture}}
    \end{center}

is isomorphic to $S_r \wr C_2$.
Note that $G$ has a non-trivial complete system of imprimitivity consisting of two blocks as represented below.

    \begin{center}
\resizebox{0.8\columnwidth}{!}
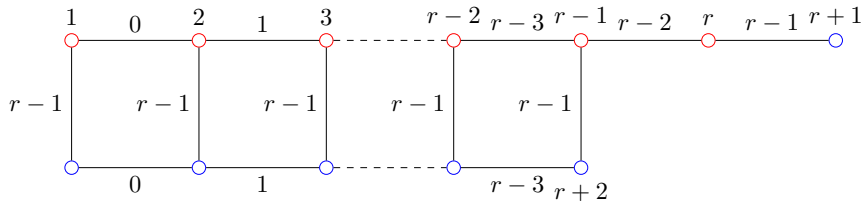
{    
    \begin{tikzpicture}
\vertex[color=red,label=above:$1$] (1) at (0,2) {};
\vertex[color=red,label=above:$2$] (2) at (2,2) {};
\vertex[color=red,label=above:$3$] (3) at (4,2) {};
\vertex[color=red,label=above:$r-2$] (4) at (6,2) {};
\vertex[color=red,label=above:$r-1$] (5) at (8,2) {};
\vertex[color=red,label=above:$r$] (6) at (10,2) {};
\vertex[color=blue,label=above:$r+1$] (7) at (12,2) {};

\vertex[color=blue] (8) at (0,0) {};
\vertex[color=blue] (9) at (2,0) {};
\vertex[color=blue] (10) at (4,0) {};
\vertex[color=blue] (11) at (6,0) {};
\vertex[color=blue,label=below:$r+2$] (12) at (8,0) {};

\draw[-] (1) -- (2)
 node[pos=0.5,above] {$0$};
\draw[-] (3) -- (2)
 node[pos=0.5,above] {$1$};
\draw[dashed] (3) -- (4);
\draw[-] (5) -- (4)
 node[pos=0.5,above] {$r-3$};
\draw[-] (6) -- (5)
 node[pos=0.5,above] {$r-2$};
\draw[-] (7) -- (6)
 node[pos=0.5,above] {$r-1$};
 
\draw[-] (9) -- (8)
 node[pos=0.5,below] {$0$}; 
\draw[-] (9) -- (10)
 node[pos=0.5,below] {$1$}; 
\draw[dashed] (10) -- (11);
\draw[-] (11) -- (12)
 node[pos=0.5,below] {$r-3$};

\draw[-] (1) -- (8)
 node[pos=0.5,left] {$r-1$};
\draw[-] (2) -- (9)
 node[pos=0.5,left] {$r-1$};
\draw[-] (3) -- (10)
 node[pos=0.5,left] {$r-1$};
\draw[-] (4) -- (11)
 node[pos=0.5,left] {$r-1$};
\draw[-] (5) -- (12)
 node[pos=0.5,left] {$r-1$};
\end{tikzpicture}}
    \end{center}

In fact, the first $r-1$ generators fix each block while $\rho_{r-1}$ permutes them. By Theorem \ref{wreath}, $G\leq S_r\wr C_2$ (since $S_r$ is the group induced by $G$ on the red block by Lemma \ref{simplices} and $\langle\rho_{r-1}\rangle=C_2$). As $G$ induces $S_r$ on the red block, for any $i\in\{1,...,r-2\}$, there is $\sigma\in G$ such that $(i,i+1)=\rho_{r-2}^{\sigma}$. Moreover, $(r+1,r+2)=\rho_{r-2}^{\rho_{r-1}}$. Thus $G$ is the full wreath product $S_r\wr C_2$, as claimed.
\end{proof}

\begin{proposition}\label{workswithsimplices}

For any integers $i\geq 2$ and $r\geq i+1$, the following permutation representation graph is a CPR graph. The group it represents is  $S_{r+i+1}$ when $r> i+1$ and $S_r\wr C_2$ when $r=i+1$.

    \begin{center}
\resizebox{\columnwidth}{!}
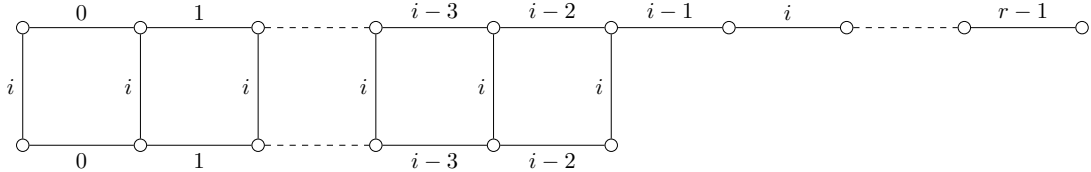
{
\begin{tikzpicture}
\vertex[] (1) at (0,2) {};
\vertex[] (2) at (2,2) {};
\vertex[] (3) at (4,2) {};
\vertex[] (4) at (6,2) {};
\vertex[] (5) at (8,2) {};
\vertex[] (6) at (10,2) {};
\vertex[] (7) at (12,2) {};
\vertex[] (8) at (14,2) {};

\vertex[] (10) at (0,0) {};
\vertex[] (11) at (2,0) {};
\vertex[] (12) at (4,0) {};
\vertex[] (13) at (6,0) {};
\vertex[] (14) at (8,0) {};
\vertex[] (15) at (10,0) {};


\vertex[] (16) at (16,2) {};
\vertex[] (17) at (18,2) {};

\draw[-] (1) -- (2)
 node[pos=0.5,above] {$0$};
\draw[-] (3) -- (2)
 node[pos=0.5,above] {$1$};
\draw[dashed] (3) -- (4);
\draw[-] (5) -- (4)
 node[pos=0.5,above] {$i-3$};
\draw[-] (5) -- (6)
 node[pos=0.5,above] {$i-2$};
\draw[-] (7) -- (6)
 node[pos=0.5,above] {$i-1$};
\draw[-] (7) --(8)
 node[pos=0.5,above] {$i$};

\draw[-] (10) -- (11)
 node[pos=0.5,below] {$0$};
\draw[-] (12) -- (11)
 node[pos=0.5,below] {$1$};
\draw[dashed] (12) -- (13);
\draw[-] (14) -- (13)
 node[pos=0.5,below] {$i-3$};
\draw[-] (14) -- (15)
 node[pos=0.5,below] {$i-2$};

\draw[-] (1) -- (10)
 node[pos=0.5,left] {$i$};
\draw[-] (2) -- (11)
 node[pos=0.5,left] {$i$};
\draw[-] (3) -- (12)
 node[pos=0.5,left] {$i$};
\draw[-] (4) -- (13)
 node[pos=0.5,left] {$i$};
\draw[-] (5) -- (14)
 node[pos=0.5,left] {$i$};
\draw[-] (6) -- (15)
 node[pos=0.5,left] {$i$};

\draw[dashed] (8) -- (16);
\draw[-] (16) -- (17)
 node[pos=0.5,above] {$r-1$};

\end{tikzpicture}}
\end{center}
\end{proposition}

\begin{proof}
    When $r=i+1$, this is true by Lemma \ref{speccase}. When $r>i+1$, the graph is obtained from its subgraph spanned by its edges of labels at most $i$ (a CPR graph by Lemma \ref{speccase}) by successive edge-additions, as described in Lemma \ref{basecase1}. Hence our graph is a CPR graph and the permutation group it represents is the symmetric group on its vertices.
\end{proof}

\section{Conclusion}

As already mentioned above, the graph met in Proposition \ref{counterexample2} was presented in~\cite[Table 6]{woof} as graph $(X)$, one of the possible permutation representation graphs of rank at least $n/2$, having a fracture graph and a split that is not perfect for $S_n$. 
In~\cite[Section 5.2]{woof}, the authors conjecture that all graphs appearing in~\cite[Table 6]{woof} are permutation representation graphs of a string C-group.
We just proved here that graph (X) of that table does not, disproving their conjecture.
In another paper, we will analyze all remaining graphs of~\cite[Table 6]{woof} and determine which ones give string C-groups and which ones do not.

As mentioned in~\cite{woof}, the number of string C-groups of rank $n-k$, with $n\geq 2k +3$, of gives the following sequence of integers indexed by $k$ and starting at $k = 1$:
 $$(1, 1, 7, 9, 35, 48, 135)$$ 
This sequence is available as sequence number A359367 in the On-Line Encyclopedia of Integer Sequences.
We believe that the gluing method described in Theorem~\ref{taping} might be useful to determine the full sequence.

\end{document}